\documentclass[12pt]{article}
\usepackage{amscd,amssymb,latexsym,amsmath,amsthm,txfonts}
\usepackage{graphicx,graphics,times,shapepar}
\usepackage{cite}
\usepackage{makeidx}
\usepackage{hyperref}
\makeindex
\newtheorem{theorem}{Theorem}[section]
\newtheorem{definition}{Definition}[section]

\newtheorem{example}{Example}[section]
\numberwithin{equation}{section}
\begin{document}
	{\title{Invariant subspace method: a tool for solving fractional partial differential equations. }
		\author{Sangita Choudhary and Varsha Daftardar-Gejji\footnote{Author for correspondence}\\
			Department of Mathematics\\
			Savitribai Phule Pune University, Pune 411007.\\
			\textit{schoudhary1695@gmail.com,\, vsgejji@gmail.com,}\\
			\textit{vsgejji@math.unipune.ac.in}}
		\date{}
		\maketitle }
	%%%% Abstract %%%%%%%%%%%%%%%%%%%%%%%%%
	\begin{abstract}
		
		In this paper invariant subspace method has been employed for solving linear and non-linear time and space fractional partial differential equations involving Caputo derivative. A variety of illustrative examples are solved to demonstrate the effectiveness and applicability of the method.
		
		\medskip
	
		
	\end{abstract}

	\section{Introduction}

	Fractional Calculus and in particular fractional differential equations (FDEs) have received considerable attention during last few decades owing to their applicability to various branches of science and engineering. FDEs have successfully modelled phenomena related to diffusion, control theory, visco-elastic systems, signal processing, fractional-order multipoles in electromagnetism and so on
	\cite{debnath2003recent,podlubny1998fractional}.
	
	Numerous decomposition/ numerical methods such as Adomian decomposition method \cite{adomian2013solving,daftardar2008solvingmulti}, New iterative method (NIM) \cite{daftardar2008solving}, Homotopy perturbation method \cite{he1999homotopy} has been employed to solve FDEs. The solutions obtained by all these methods however are local in nature and it is important to explore other techniques to find exact analytical solutions of FDEs. Exact	solutions play a vital role in the proper understanding of qualitative features of the concerned phenomena and processes	in various areas of science and engineering. 
	
	One of the analytical methods for finding exact solutions of partial differential equations (PDEs) is invariant subspace method. This method was proposed by Galaktionov and Svirshchevskii \cite{galaktionov1995invariant,galaktionov2006exact}and have been applied by many researchers thereafter \cite{ma2012refined,svirshchevskii1995lie,svirshchevskii1996invariant}. This method has been further extended for time fractional partial differential equations (FPDEs) and its utility has been illustrated in the literature \cite{gazizov2013construction,harris2013analytic,harris2014nonlinear,sahadevan2015invariant,sahadevan2016exact}. In the present paper we develop invariant subspace method for finding exact solutions of PDEs with fractional time and space derivatives. In this method fractional partial differential equations are reduced to systems of fractional differential equations in one variable which can be solved by known analytical methods.
	
	The organization of the paper is as follows. Section 2, deals with preliminaries and notations. In Section 3, we develop invariant subspace method for solving time and space fractional partial differential equations. In Section 4, we present illustrative examples to explain applicability of the method. It is clear from the examples that the method also works for time and space FPDEs with certain kinds of initial conditions. Finally concluding remarks are made in Section 5.
	
	\section{Preliminaries and Notations}
	In this section, we introduce notations, definitions and preliminaries which are used in the paper \cite{diethelm2010analysis,podlubny1998fractional}.
	\begin{definition}The fractional integral of order $\alpha > 0$ of function $f\in L^1([a,b],\mathbb{R})$ is defined as
		\begin{equation*}
		I^{\alpha}f(t)=\frac{1}{\Gamma(\alpha)}\int\limits_{a}^{t}\frac{f(\tau)}{(t-\tau)^{1-\alpha}}d\tau,~~t>a.
		\end{equation*}
	\end{definition}
	\begin{definition}
		Caputo fractional derivative of order $\alpha>0$ of $f\in C^n([a,b],\mathbb{R})$ is defined as 
		\begin{align*}
		\dfrac{\rm d^{\alpha}f(t)}{\rm dt^\alpha}=\left\{\begin{array}{lcr}I^{n-\alpha}D^{n}f(t)=\dfrac{1}{\Gamma(n-\alpha)}\displaystyle\int\limits_{a}^{t}\dfrac{f^{(n)}(\tau)}{(t-\tau)^{\alpha-n+1}}d\tau,~\mbox{$n-1 < \alpha < n,$}\\f^{(n)}(t),\hspace{6.4cm} \mbox{$\alpha=n,~n \in \mathbb{N}.$}
		\end{array}\right.
		\end{align*}
	\end{definition}
	\iffalse{
		\noindent	Riemann-Liouville fractional integral and Caputo fractional derivative satisfy the following properties\cite{diethelm2010analysis}:
		\begin{eqnarray*}
			I^\alpha t^\beta&=&\dfrac{\Gamma(\beta+1)}{\Gamma(\beta+\alpha+1)}t^{\beta+\alpha},\hskip 1cm \alpha >0,~\beta>-1,~t>0.\\
			\dfrac{\rm d^{\alpha}t^\beta}{\rm dt^\alpha} &=&\left\{\begin{array}{lcr}
				0,\hskip 3.2cm\text{if}~ \alpha >0,~\beta\in \{0,1,2,\dots,n-1\},\vspace{.1cm}\\
				\dfrac{\Gamma(\beta+1)}{\Gamma(\beta-\alpha+1)}t^{\beta-\alpha},\hskip .6cm \mbox{if $\alpha >0,~\beta\in \mathbb{N}$ and $\beta \geq n $ or $\beta\notin \mathbb{N}$ and $\beta>n-1$}.
			\end{array}\right.\\
			I^\alpha \left(\dfrac{\rm d^{\alpha} f(t)}{\rm dt^\alpha}\right)&=&f(t)-\sum_{k=0}^{n-1}\dfrac{\rm d^{(k)}f(0)}{\rm dt^{(k)}}\dfrac{t^k}{k!},~~~~~~n-1<\alpha<n,~t>0.
		\end{eqnarray*}
	}\fi
	The $n-$th order derivative of two-parametric Mittag-Leffler function $E_{\alpha,\beta}(t)$  is given by
	\begin{equation}\label{D1}
	E_{\alpha,\beta}^{(n)}(t)=\dfrac{\rm d^{n}}{\rm dt^{n}}E_{\alpha,\beta}(t)=\displaystyle\sum\limits_{k=0}^{\infty}\dfrac{(k+n)! t^{k}}{k! \Gamma(\alpha k+\alpha n+\beta)},~~n=0,1,2,\ldots.
	\end{equation}
	
	Laplace transform of the Caputo derivative of order $\alpha$ is,
	\begin{equation*}\mathcal{L}\left\{\dfrac{	\rm d^\alpha f(t)}{\rm dt^\alpha};s\right\}=s^{\alpha}\tilde{f}(s)-\sum\limits_{k=0}^{n-1}s^{\alpha-k-1}f^{(k)}(0),~~n-1<\alpha\leq{n},~n\in\mathbb{N},~\mathfrak{R}(s)>0,\end{equation*}
	where~$\tilde{f}(s)=\mathcal{L}\{f(t);s\}=\int\limits_{0}^{\infty}e^{-st}f(t)dt,~~s\in\mathbb{R}.$\\
	%%%%%%%%%%%%%%%%%%%%%%%%%%%%%%%%%%%%%%%%%%%%%%%%
	
	Laplace transform of the function $\varepsilon_n(t,a;\alpha,\beta):=t^{\alpha n+\beta-1}E_{\alpha,\beta}^{(n)}(\pm at^{\alpha})$ has the form 
	\begin{equation}\label{D2}
	\mathcal{L}\{t^{\alpha n+\beta-1}E_{\alpha,\beta}^{(n)}(\pm at^{\alpha});s\}=\frac{n!s^{\alpha-\beta}}{(s^{\alpha}\mp a)^{n+1}},~~\mathfrak R(s)>|a|^{\frac{1}{\alpha}},~n=0,1,\ldots. \end{equation}
	%%%%%%%%%%%%%%%%%%%%%%%%%%%%%%%%%%%%%%%%%%%%%%%%%%%
	
	Let $I_n$ denote the n-dimensional linear space over $\mathbb{R}$ spanned by n linearly independent functions $\{\phi_i(x):~ i=0,1,\dots,n-1\}$, \textit{i.e.},
	\begin{eqnarray*}
		\hskip.4cm I_n= \mathfrak{L}\{\phi_0(x),\dots,\phi_{n-1}(x)\}=\left \{\sum_{i=0}^{n-1}k_i \phi_i(x)~ \bigg| ~k_i\in \mathbb{R}, i=0,\dots,n-1 \right \}
	\end{eqnarray*}
	
	A finite dimensional linear space $I_n$ is invariant with respect to a differential operator $N$ if $N[f]\in I_n, \forall f \in I_n.$

	\section{Main Results}

	Consider the following fractional partial differential equation 
	\noindent \begin{eqnarray}\label{FPDE}
	\sum_{i=0}^{m}\lambda_i\dfrac{\partial ^{\alpha+i}f(x,t)}{\partial t^{\alpha+i}}=N[f(x,t)]
	=N\left(x,f,\dfrac{\partial^{\beta}f}{\partial x^{\beta}},\dots,\dfrac{\partial^{\beta+k}f}{\partial x^{\beta+k}}\right),\end{eqnarray}
	where $N[f]$ is linear/ nonlinear differential operator. $\dfrac{\partial^{\alpha+i}f}{\partial t^{\alpha+i}},i=0,1,\dots,m$ and $\dfrac{\partial^{\beta+j}f}{\partial x^{\beta+j}},~j=0,\dots,k$ are Caputo time and space derivatives respectively. $\lfloor\alpha\rfloor= r,\lfloor \beta \rfloor =s, \textrm{where}~ r,s \in \mathbb{N}$ and $\lambda_i\in \mathbb{R}.$
	
	\begin{theorem}\label{THM1}
		If a finite dimensional linear space $I_{n+1}=\mathfrak{L}\{\phi_j(x)~|~j=0,\dots,n\}$ is invariant under the operator N[f], then FPDE (\ref{FPDE}) has a solution of the form
		\begin{equation}\label{U}
		f(x,t)=\sum_{j=0}^{n}K_j(t)\phi_j(x),
		\end{equation}where the coefficients $\{K_j\}$ satisfy the following system of FDEs
		\begin{equation*}
		\sum_{i=0}^{m}\lambda_i\dfrac{\rm d^{\alpha+i}K_j(t)}{\rm dt^{\alpha+i}}=\psi_j(K_1(t),K_2(t),\dots,K_n(t)),~j=0,1,\dots,n.
		\end{equation*}Here $\psi_j$'s are the expansion coefficients of $N[f]$ with respect to basis $\phi_j$'s.
	\end{theorem}
	\begin{proof} Consider L.H.S of FPDE (\ref{FPDE}). Using linearity of fractional derivatives and Eq. (\ref{U}), we get
	\begin{eqnarray}\label{LHS}
	\sum_{i=0}^{m}\lambda_i\dfrac{\partial ^{\alpha+i}f(x,t)}{\partial t^{\alpha+i}}&=&\sum_{i=0}^{m}\lambda_i\frac{\partial^{\alpha+i}}{\partial t^{\alpha+i}}\sum_{j=0}^{n}K_j(t) \phi_j(x)\nonumber\\
	&=&\sum_{j=0}^{n}\left[\sum_{i=0}^{m}\lambda_i\frac{\rm d^{\alpha+i}K_j(t)}{\rm dt^{\alpha+i}}\right] \phi_j(x).
	\end{eqnarray}	
	Further as $I_{n+1}$ is invariant space under the operator $N[f]$, there exist $n+1$ functions $\psi_0,\psi_1,\dots,\psi_n$ such that \begin{equation}\label{IN}
	N\left[\sum_{j=0}^{n}k_j \phi_j(x)\right]=\sum_{j=0}^{n}\psi_j(k_1,k_2,\dots,k_n)\phi_j(x),~\text{for}~ k_j\in \mathbb{R},\end{equation}
	where $\psi_j$'s are expansion coefficients of $N[f]\in I_{n+1}$ corresponding to basis functions $\phi_j$'s.\\
	In view of Eq. (\ref{U}) and Eq. (\ref{IN})
	\begin{equation}\label{RHS}
	N[f(x,t)]=N\left[\sum_{j=0}^{n}K_j(t)\phi_j(x)\right]=\sum_{j=0}^{n}\psi_j(K_1(t),\dots,K_n(t))\phi_j(x).
	\end{equation}
	Substituting Eq. (\ref{LHS}) and Eq. (\ref{RHS}) in Eq. (\ref{FPDE}), we get
	\begin{equation}\label{LHS=RHS}
	\sum_{j=0}^{n}\left[\sum_{i=0}^{m}\lambda_i\dfrac{\rm d^{\alpha+i}K_j(t)}{\rm dt^{\alpha+i}}-\psi_j(K_1(t),K_2(t),\dots,K_n(t))\right]\phi_j(x)=0.\end{equation}
	Using Eq. (\ref{LHS=RHS}) and the fact that $\phi_j$'s are linearly independent, we get the following system of FDEs
	\begin{equation*}
	\sum_{i=0}^{m}\lambda_i\dfrac{\rm d^{\alpha+i}K_j(t)}{\rm dt^{\alpha+i}}=\psi_j(K_1(t),K_2(t),\dots,K_n(t)),~j=0,1,\dots,n.
	\end{equation*} 
\end{proof}
	Consider the following fractional PDE 
	\begin{equation}\label{FPDE2}
	\sum_{i=1}^{m}\lambda_i\dfrac{\partial ^{i\alpha}f(x,t)}{\partial t^{i\alpha}}=N[f(x,t)]=N\left(x,f,\dfrac{\partial^{\beta}f}{\partial x^{\beta}},\dots,\dfrac{\partial^{k\beta}f}{\partial x^{k\beta}}\right),\end{equation}
	where $N[f]$ is linear/ nonlinear differential operator. $\dfrac{\partial^{i\alpha}f}{\partial t^{i\alpha}},~i=1,2,\dots,m$ and $\dfrac{\partial^{j\beta}f}{\partial x^{j\beta}}, ~j=1,2,\dots,k$ are Caputo Caputo time and space derivatives respectively. $\lfloor\alpha\rfloor= r,\lfloor \beta \rfloor =s, \textrm{where}~ r,s \in \mathbb{N}$ and $\lambda_i\in \mathbb{R}.$
	
	\begin{theorem}\label{THM2}
		If a finite dimensional linear space $I_n$ is invariant under the operator N[f], then FPDE (\ref{FPDE2}) has a solution of the form
		\begin{equation*}
		f(x,t)=\sum_{j=1}^{n}K_j(t)\phi_j(x),
		\end{equation*}where the coefficients $\{K_j\}$ satisfy the following system of FDEs
		\begin{equation*}
		\sum_{i=1}^{m}\lambda_i\dfrac{\rm d^{i\alpha}K_j(t)}{\rm dt^{i\alpha}}=\psi_j(K_1(t),K_2(t),\dots,K_n(t)),~j=1,2,\dots,n.
		\end{equation*}Here $\psi_j$'s are expansion coefficients of $N[f]\in I_n$ with respect to $\phi_j$'s.
	\end{theorem}
	Proof of Theorem \ref{THM2} can be given on similar lines as the proof of Theorem \ref{THM1}.

	\section{Illustrative Examples}
	In this section we solve some illustrative examples using Theorem \ref{THM1} and Theorem \ref{THM2}.
	\begin{example}
		Consider the following fractional PDE
		\begin{equation}\label{EX1}
		\dfrac{\partial ^\alpha f}{\partial t^ {\alpha}}+\dfrac{\partial ^{\alpha+1} f}{\partial t^ {\alpha+1}}+\dfrac{\partial ^{\alpha+2} f}{\partial t^ {\alpha+2}}=N[f]=\dfrac{\partial ^{\beta+1} f}{\partial x^ {\beta+1}},~~\alpha,~\beta \in (0,1].
		\end{equation}	
	\end{example}
	\noindent Clearly $I_2=\mathfrak{L}\{1,x^{\beta+1}\}$ is an invariant subspace as
	\begin{equation*}
	N[k_0+k_1x^{\beta+1}]=k_1\Gamma(\beta+2)\in I_2.
	\end{equation*}In view of Theorem \ref{THM1}, Eq. (\ref{EX1}) has exact solution of the form
	\begin{equation*}\label{SOL1}
	f(x,t)=K_0(t)+K_1(t)x^{\beta+1},
	\end{equation*} where $K_0(t)$ and $K_1(t)$ satisfy the following system of FDEs.
	\begin{eqnarray}
	\dfrac{\rm d ^\alpha K_0(t)}{\rm d t^ {\alpha}}+\dfrac{\rm d ^{\alpha+1} K_0(t)}{\rm d t^ {\alpha+1}}+\dfrac{\rm d ^{\alpha+2} K_0(t)}{\rm d t^ {\alpha+2}}&=&K_1(t)\Gamma(\beta+2),\label{11}\\
	\dfrac{\rm d ^\alpha K_1(t)}{\rm d t^ {\alpha}}+\dfrac{\rm d ^{\alpha+1} K_1(t)}{\rm d t^ {\alpha+1}}+\dfrac{\rm d ^{\alpha+2} K_1(t)}{\rm d t^ {\alpha+2}}&=&0.\label{12}
	\end{eqnarray}
	Using Laplace transform in Eq. (\ref{12}), we obtain
	\begin{align}
	\tilde{K_1}(s)&=\frac{K_1(0)(s^{\alpha+1}+s^{\alpha}+s^{\alpha-1})}{s^{\alpha+2}+s^{\alpha+1}+s^{\alpha}}+\frac{K_1'(0)(s^{\alpha}+s^{\alpha-1})}{s^{\alpha+2}+s^{\alpha+1}+s^{\alpha}}+\frac{K_1''(0)s^{\alpha-1}}{s^{\alpha+2}+s^{\alpha+1}+s^{\alpha}}\nonumber\\
	&=K_1(0)\left[\sum_{k=0}^{\infty}\frac{(-1)^ks^{-k}}{(s+1)^{k+1}}+\sum_{k=0}^{\infty}\frac{(-1)^ks^{-k-1}}{(s+1)^{k+1}}+\sum_{k=0}^{\infty}\frac{(-1)^ks^{-k-2}}{(s+1)^{k+1}}\right]+\nonumber\\
	&K_1'(0)\left[\sum_{k=0}^{\infty}\frac{(-1)^ks^{-k-1}}{(s+1)^{k+1}}+\sum_{k=0}^{\infty}\frac{(-1)^ks^{-k-2}}{(s+1)^{k+1}}\right]+K_1''(0)\sum_{k=0}^{\infty}\frac{(-1)^ks^{-k-2}}{(s+1)^{k+1}}.\nonumber
	\end{align}
	Taking inverse Laplace transform and using Eq. (\ref{D2}) we get
	\begin{eqnarray}\label{13}
	K_1(t)=\sum_{i=0}^{2}K_1^{(i)}(0)\left[\sum_{l=i}^{2}\sum_{k=0}^{\infty}\frac{(-1)^k}{k!}t^{2k+l}E^{(k)}_{1,k+l+1}(-t)\right],
	\end{eqnarray}
	where $E^{(n)}_{\alpha, \beta}(.)$ is defined in Eq. (\ref{D1}).\\
	Substituting value of $K_1(t)$ from Eq. (\ref{13}) in Eq. (\ref{11}) and following similar procedure we deduce
	\begin{align}\label{14}
	K_0(t)=\sum_{i=0}^{2}K_0^{(i)}(0)\left[\sum_{l=i}^{2}\sum_{k=0}^{\infty}\frac{(-1)^k}{k!}t^{2k+l}E^{(k)}_{1,k+l+1}(-t)\right]+\Gamma(\beta+2)\nonumber\\\left[
	\sum_{i=0}^{2}K_1^{(i)}(0)\left(\sum_{l=i+2}^{4}\sum_{k=0}^{\infty}\frac{(-1)^k}{k!}t^{2k+\alpha+l}E^{(k)}_{1,k+\alpha+l}(-t)\right)\right].
	\end{align}
	Hence an exact solution of Eq. (\ref{EX1}) is~$f(x,t)=K_0(t)+K_1(t)x^{\beta+1}$~where $K_0(t)$ and $K_1(t)$ are given by Eq. (\ref{14}) and Eq. (\ref{13}) respectively.

	%EX2*************************************************
	
	\begin{example}
		Consider the following fractional generalization of time fractional Burgers' equation.
		\begin{equation}\label{EX2}
		\dfrac{\partial^\alpha f}{\partial t^\alpha}+f\dfrac{\partial^\beta f}{\partial x^\beta}-d\dfrac{\partial^\beta }{\partial x^\beta}\left(\dfrac{\partial^\beta f}{\partial x^\beta}\right)=0,~~\alpha\in(0,1),~\alpha\neq \frac{1}{2},~\beta\in (0,1],		
		\end{equation}where $d$ is diffusion constant.
	\end{example}
	\noindent Here $N[f]=-f\dfrac{\partial^\beta f}{\partial x^\beta}+d\dfrac{\partial^\beta }{\partial x^\beta}\left(\dfrac{\partial^\beta f}{\partial x^\beta}\right).~I_2=\mathfrak{L}\{1,x^\beta\}$ is an invariant subspace as $
	N[k_0+k_1x^\beta]=-k_0k_1\Gamma(\beta+1)-k_1^2\Gamma(\beta+1)x^\beta \in I_2.$\\
	Theorem \ref{THM1} implies an exact solution of the form
	\begin{equation*}\label{SOL21}
	f(x,t)=K_0(t)+K_1(t)x^\beta,
	\end{equation*} where $K_0(t)$ and $K_1(t)$ are unknown functions to be determined by solving following system of equations.
	\begin{eqnarray}
	\dfrac{\rm d^\alpha K_0(t)}{\rm d t^\alpha}&=&-K_0(t)K_1(t)\Gamma(\beta+1),\label{21}\\
	\dfrac{\rm d^\alpha K_1(t)}{\rm d t^\alpha}&=&-K_1(t)^2\Gamma(\beta+1).\label{22}
	\end{eqnarray}
	Assume solution of Eq. (\ref{22}) of the form $K_1(t)=\eta t^\gamma,$ where constant $\eta $ and the exponent $\gamma>-1$ can be found directly by substituting $K_1(t)=\eta t^\gamma$ in (\ref{22}), \textit{i.e.},
	\begin{equation}
	\dfrac{\rm d^\alpha K_1(t)}{\rm d t^\alpha}=\eta \dfrac{\Gamma(\gamma+1)}{\Gamma(\gamma-\alpha+1)}t^{\gamma-\alpha}=-\eta ^2t^{2\gamma}\Gamma(\beta+1).
	\end{equation}
	Thus $K_1(t)=\eta t^\gamma$ is a solution if $\gamma=-\alpha$ and $\eta =\dfrac{-\Gamma(1-\alpha)}{\Gamma(1+\beta)\Gamma(1-2\alpha)},$ \textit{i.e.},
	\begin{equation*}
	K_1(t)=-\dfrac{\Gamma(1-\alpha)}{\Gamma(1+\beta)\Gamma(1-2\alpha)}t^{-\alpha},~~\alpha \neq\frac{1}{2},~\alpha \neq 1.
	\end{equation*}
	On similar lines by solving Eq. (\ref{21}) we get
	$
	K_0(t)=at^{-\alpha},
	$	where $a$ is arbitrary constant. Therefore Eq. (\ref{EX2}) has following exact solution.
	\begin{equation*}
	f(x,t)=at^{-\alpha}-\left[\frac{\Gamma(1-\alpha)t^{-\alpha}}{\Gamma(1+\beta)\Gamma(1-2\alpha)}\right]x^{\beta}.
	\end{equation*}
	In particular, for $\beta=1$ we obtain\\
	\begin{equation*}
	f(x,t)=at^{-\alpha}-\left[\frac{\Gamma(1-\alpha)}{\Gamma(1-2\alpha)}t^{-\alpha}\right]x,
	\end{equation*} which is a solution of the time fractional Burgers equation \cite{sahadevan2016exact}.\\
	A particular case of fractional Burgers equation with diffusion constant $d=0$ has been studied by Harris and Garra \cite{harris2013analytic}, in which they have considered an exact solution of the form $f(x,t)=\textrm{(constant)}+C_1(t)x^\beta.$
	% EX3******************************************************************
	\begin{example}
		Consider the time and space fractional diffusion equation
		\begin{equation}\label{EX3}
		\dfrac{\partial^{\alpha} f}{\partial t^{\alpha}}=C\dfrac{\partial^\beta f}{\partial x^\beta},~~\alpha\in (0,1],~\beta \in (1,2].
		\end{equation}
	\end{example}
	\noindent Here $N[f]=C\dfrac{\partial^\beta f}{\partial x^\beta}.$ Subspace $I_3=\mathfrak{L}\{1,x^\beta,x^{2\beta}\}$ is invariant under N[f] as
	\begin{equation*}
	N[k_0+k_1x^\beta+k_2x^{2\beta}]=Ck_1\Gamma(\beta+1)+Ck_2\frac{\Gamma(2\beta+1)}{\Gamma(\beta+1)}x^\beta\in I_3.
	\end{equation*}
	Consider an exact solution of the form	$
	f(x,t)=K_0(t)+K_1(t)x^\beta+K_2(t)x^{2\beta},
	$ where $K_0(t)$ and $K_1(t)$ satisfy the following system of FDEs:
	\begin{eqnarray}
	\dfrac{\rm d^\alpha K_0(t)}{\rm d t^\alpha}&=&C~\Gamma(\beta+1)K_1(t)\label{31},\\
	\dfrac{\rm d^\alpha K_1(t)}{\rm d t^\alpha}&=&\frac{C~\Gamma(2\beta+1)}{\Gamma(\beta+1)}K_2(t)\label{32},\\
	\dfrac{\rm d^\alpha K_2(t)}{\rm d t^\alpha}&=&0.\label{33}
	\end{eqnarray}
	Eq. (\ref{33}) implies that $K_2(t)=a$ (say). So Eq. (\ref{32}) takes the form
	\begin{equation*}
	\dfrac{\rm d^\alpha K_1(t)}{\rm d t^\alpha}=aC\dfrac{\Gamma(2\beta+1)}{\Gamma(\beta+1)},
	\end{equation*}which has the following solution:
	\begin{eqnarray*}
		K_1(t)&=&K_1(0)+\frac{aC~\Gamma(2\beta+1)}{\Gamma(\beta+1)\Gamma(\alpha+1)}t^\alpha.\end{eqnarray*}
	Similarly Eq. (\ref{31}) yields
	\begin{eqnarray*} K_0(t)&=&K_0(0)+K_1(0)\frac{C~\Gamma(\beta+1)}{\Gamma(\alpha+1)}t^\alpha +\frac{aC^2~\Gamma(2\beta +1)}{\Gamma(2\alpha+1)} t^{2\alpha}.
	\end{eqnarray*}
	Hence Eq. (\ref{EX3}) has the following solution
	\begin{eqnarray}\label{SOL32}
	f(x,t)=\left[d+\frac{bC~\Gamma(\beta+1)}{\Gamma(\alpha+1)}t^\alpha +\frac{aC^2~\Gamma(2\beta +1)}{\Gamma(2\alpha+1)} t^{2\alpha}\right]\nonumber\\
	+\left[b+\frac{aC~\Gamma(2\beta+1)}{\Gamma(\beta+1)\Gamma(\alpha+1)}t^\alpha\right]x^\beta+ax^{2\beta},	
	\end{eqnarray}
	where~$a, b$ and $d$ are arbitrary constants.\\	
	
	\noindent It can  be easily verified that $I_2=\mathfrak{L}\{1,E_\beta (x^\beta )\}$ is also an invariant subspace with respect to $N[f]$ leading to a different solution having the form 
	\begin{equation*}
	f(x,t)=K_0(t)+K_1(t)E_\beta (x^\beta ),
	\end{equation*}where unknown functions $K_0(t)$ and $K_1(t)$ can be found by solving the following system of equations:
	\begin{eqnarray}
	\dfrac{\rm d^\alpha K_0(t)}{\rm d t^\alpha}&=&0\label{34},\\
	\dfrac{\rm d^\alpha K_1(t)}{\rm d t^\alpha}&=&CK_1(t).\label{35}
	\end{eqnarray}
	Solving the system of equations (\ref{34})-(\ref{35}), we obtain an exact solution of Eq. (\ref{EX3}) associated with $I_2=\mathfrak{L}\{1,E_\beta (x^\beta )\}$ as
	\begin{equation}\label{SOL33}
	f(x,t)=a+[bE_\alpha (Ct^\alpha )]E_\beta (x^\beta ),
	\end{equation}where $a$ and $b$ are arbitrary constants.\\
	
	\noindent Note that Eq. (\ref{EX3}) also admits another invariant subspace $I_2=\mathfrak{L}\{1,x^\beta\}$, which yields to the exact solution of Eq. (\ref{EX3}):
	\begin{equation}\label{SOL34}
	f(x,t)=a+\frac{bC\Gamma(\beta+1)}{\Gamma(\alpha+1)}t^\alpha+bx^{\beta},	~~a, b\in \mathbb{R}.
	\end{equation}
	Thus (\ref{SOL32}), (\ref{SOL33}) and (\ref{SOL34}) are three distinct exact solutions of Eq. (\ref{EX3}) corresponding to three distinct invariant subspaces. Note that $I_{n+1}=\mathfrak{L}\{1,x^\beta,\dots,x^{n\beta}\},$
	$~n\in \mathbb{N}$ is yet another invariant subspace for (\ref{EX3}). Thus we get infinitely many invariant subspaces for the problem under consideration, which in turn yield infinitely many solutions.\\
	
	\textbf{Solutions pertaining to initial conditions}\\
	
	\noindent 	Invariant subspace method is also useful in finding closed form solutions of fractional PDEs satisfying initial conditions. Observe that Eq. (\ref{EX3}) along with the initial condition \begin{equation*}
	f(x,0)=a_0+a_1x^\beta+a_2x^{2\beta}+\cdots+a_nx^{n\beta},~1<\beta \leq 2,~a_i \in \mathbb{R},~i=0,1,\dots n,
	\end{equation*}has closed form solutions.\\
	
	For example, consider the time and space fractional diffusion equation (\ref{EX3}) along with the  initial condition,
	\begin{equation}\label{IC31}
	f(x,0)=\dfrac{3x^\beta}{\Gamma(\beta+1)}.
	\end{equation}
	For solving this IVP, consider the invariant subspace $I_2=\mathfrak{L}\{1,x^\beta\}$. Using initial condition (\ref{IC31}) in Eq. (\ref{SOL34}) we obtain $a=0$ and $b=\dfrac{3}{\Gamma(\beta+1)}$.\\
	Thus we obtain the following exact solution of the IVP (\ref{EX3})-(\ref{IC31}).
	\begin{equation}\label{SOL35}
	f(x,t)=\dfrac{3Ct^\alpha}{\Gamma(\alpha+1)}+\dfrac{3x^\beta}{\Gamma(\beta+1)}.
	\end{equation}
	Note that the solution (\ref{SOL35}) coincides with the solution obtained by NIM introduced by Daftardar-Gejji \textit{et al.} \cite{daftardar2008solving}.\\
	Now consider (\ref{EX3}) together with the initial condition,
	\begin{equation}\label{IC32}
	f(x,0)=1-\dfrac{x^{2\beta}}{2}.
	\end{equation}Using the invariant subspace $I_3=\mathfrak{L}\{1,x^\beta,x^{2\beta}\}$, we get the following exact solution of the IVP (\ref{EX3}) - (\ref{IC32}).
	\begin{equation*}
	f(x,t)=\left[1-\frac{C^2~\Gamma(2\beta +1)}{2~\Gamma(2\alpha+1)} t^{2\alpha}\right]-\left[\frac{C~\Gamma(2\beta+1)}{2~\Gamma(\beta+1)\Gamma(\alpha+1)}t^\alpha\right]x^\beta-\dfrac{1}{2}x^{2\beta}.	
	\end{equation*} 
	%
	%Ex4***********************
	\begin{example}
		Consider the following fractional generalization of telegraph equation:
		\begin{equation}\label{EX4}
		\dfrac{\partial ^{\alpha+1} f}{\partial t^ {\alpha+1}}+\dfrac{\partial ^\alpha f}{\partial t^ {\alpha}}=\dfrac{\partial ^{\beta} f}{\partial x^ {\beta}}-f,~~\alpha \in (0,1],~\beta \in (1,2].
		\end{equation}	
	\end{example}
	\noindent Note that for $I_2=\mathfrak{L}\{1,E_\beta (x^\beta)\}$,
	\begin{equation*}
	N[k_0+k_1E_\beta (x^\beta)]=k_1E_\beta (x^\beta)-k_0-k_1E_\beta (x^\beta)=-k_0\in I_2.
	\end{equation*}This implies that $I_2$ is an invariant subspace, and in view of Theorem \ref{THM1}, Eq. (\ref{EX4}) has an exact solution of the form
	\begin{equation}\label{SOL41}
	f(x,t)=K_0(t)+K_1(t)E_\beta (x^\beta),
	\end{equation} where $K_0(t)$ and $K_1(t)$ satisfy the following system of FDEs.
	\begin{align}
	\dfrac{\rm d ^{\alpha+1} K_0(t)}{\rm d t^ {\alpha+1}}+\dfrac{\rm d ^\alpha K_0(t)}{\rm d t^ {\alpha}}&=-K_0(t),\label{41}\\
	\dfrac{\rm d ^{\alpha+1} K_1(t)}{\rm d t^ {\alpha+1}}+\dfrac{\rm d ^\alpha K_1(t)}{\rm d t^ {\alpha}}&=0.\label{42}
	\end{align}
	Applying Laplace transform technique to Eq. (\ref{41}) we get
	\begin{eqnarray*}
		\tilde{K_0}(s)=K_0(0)\sum_{k=0}^{\infty}\frac{(-1)^ks^{-\alpha k}}{(1+s)^{k+1}}+[K_0(0)+K_0'(0)]\sum_{k=0}^{\infty}\frac{(-1)^ks^{-\alpha k-1}}{(1+s)^{k+1}}.
	\end{eqnarray*}
	Taking inverse Laplace transform we obtain
	\begin{align*}
	K_0(t)=~~&	K_0(0)\sum_{k=0}^{\infty}\frac{(-1)^k}{k!}t^{k(\alpha+1)}E_{1,\alpha k +1}^{(k)}(-t)+[K_0(0)+K_0'(0)]\\
	&\sum_{k=0}^{\infty}\frac{(-1)^k}{k!}t^{k(\alpha+1)+1}E_{1,\alpha k +2}^{(k)}(-t).
	\end{align*}
	Proceeding on similar lines, from Eq. (\ref{42}) we deduce
	\begin{equation*}
	K_1(t)=K_1(0)+K_1'(0)tE_{1,2}(-t).
	\end{equation*}
	Thus an exact solution (\ref{SOL41}) takes the following form
	\begin{eqnarray*}
		f(x,t)&=&a_1\sum_{k=0}^{\infty}\frac{(-1)^k}{k!}t^{k(\alpha+1)}E_{1,\alpha k +1}^{(k)}(-t)+[a_1+a_2]\sum_{k=0}^{\infty}\frac{(-1)^k}{k!}\\
		&&t^{k(\alpha+1)+1}E_{1,\alpha k +2}^{(k)}(-t)+[b_1+b_2tE_{1,2}(-t)]E_\beta(x^\beta),
	\end{eqnarray*}where $a_1, a_2, b_1$ and $b_2$ are arbitrary constants.  
	%Ex5****************************************************************
	%********************************************************
	%*EX6****************************************************
	\begin{example}
		Consider the following fractional IVP:
		\begin{eqnarray}
		\dfrac{\partial^\alpha f}{\partial t^\alpha}&=&\left(\dfrac{\partial^\beta f}{\partial x^\beta }\right)^2-f\left(\dfrac{\partial^\beta  f}{\partial x^\beta }\right),~~\alpha,~\beta  \in (0,1],\label{EX6}\\
		f(x,0)&=&3+\frac{5}{2}E_\beta (x^\beta )\label{IC6}.                                    
		\end{eqnarray}
	\end{example}		
	\noindent$I_2=\mathfrak{L}\{1,E_\beta (x^\beta )\}$ is invariant under $N[f]=\left(\dfrac{\partial^\beta f}{\partial x^\beta }\right)^2-f\left(\dfrac{\partial^\beta  f}{\partial x^\beta }\right)$ as
	\begin{equation*}
	N[k_0+k_1E_\beta (x^\beta )]=-k_0k_1E_\beta (x^\beta )\in I_2.
	\end{equation*}
	We consider solution of the form
	\begin{equation}\label{SOL61}
	f(x,t)=K_0(t)+K_1(t)E_\beta (x^\beta ),
	\end{equation} where $K_0(t)$ and $K_1(t)$ are unknown functions to be determined.\\
	By substituting (\ref{SOL61}) in (\ref{EX6}) and equating coefficients of different powers of $x$, we get the following system of FDEs.
	\begin{eqnarray}
	\dfrac{\rm d^\alpha K_0(t)}{\rm d t^\alpha}&=&0,\label{61}\\
	\dfrac{\rm d^\alpha K_1(t)}{\rm d t^\alpha}&=&-K_0(t)K_1(t).\label{62}
	\end{eqnarray}
	Solving Eq. (\ref{61}) we get $K_0(t)=K_0(0)=a$ (say), and Eq. (\ref{62}) takes the form $\dfrac{d^\alpha K_1(t)}{d t^\alpha}=-aK_1(t)$.
	Using Laplace transform technique we obtain
	$K_1(t)=K_1(0)E_\alpha(-at^\alpha).$ Thus the exact solution of (\ref{EX6}) along with the initial condition (\ref{IC6}) turns out to be
	\begin{equation*}
	f(x,t)=3+\left[\frac{5}{2}E_\alpha(-3t^\alpha)\right]E_\beta (x^\beta ).
	\end{equation*}
	
	%E7****************************************
	\begin{example}
		Consider the following two generalization of wave equation with constant absorption term\\
		\textbf{Type (I)}
		\begin{eqnarray}\label{EX7}
		\dfrac{\partial^{2\alpha} f}{\partial t^{2\alpha}}=\dfrac{\partial^\beta }{\partial x^\beta}\left(f\dfrac{\partial^\beta f}{\partial x^\beta}\right)-1,~~\alpha,~\beta\in (0,1],\\
		f(x,0)=e+\dfrac{x^\beta}{\Gamma(\beta+1)},~f_t(x,0)=1-x^\beta.\label{IC7}
		\end{eqnarray}
	\end{example}
	\noindent Here $N[f]=\dfrac{\partial^\beta }{\partial x^\beta}\left(f\dfrac{\partial^\beta f}{\partial x^\beta}\right)-1.$ Clearly $I_2=\mathfrak{L}\{1,x^\beta\}$ is an invariant subspace of N[f] as
	\begin{equation*}
	N[k_0+k_1x^\beta]=k_1^2[\Gamma(\beta+1)]^2-1\in I_2.
	\end{equation*}
	In view of Theorem \ref{THM2}, we look for an exact solution of the form
	\begin{equation*}\label{SOL7}
	f(x,t)=K_0(t)+K_1(t)x^\beta,
	\end{equation*} where $K_0(t)$ and $K_1(t)$ satisfy the following system of FDEs.
	\begin{eqnarray}
	\dfrac{\rm d^{2\alpha }K_0(t)}{\rm d t^{2\alpha }}&=&K_1^2(t)[\Gamma(\beta+1)]^2-1,\label{71}\\
	\dfrac{\rm d^{2\alpha } K_1(t)}{\rm d t^{2\alpha }}&=&0.\label{72}
	\end{eqnarray}
	Solving the system of FDEs (\ref{71})-(\ref{72}) we get the following exact solution
	
	\begin{align}\label{SOL71}
	f(x,t)=\left\{\begin{array}{lcr}
	\left[a_1+\left(\dfrac{b_1^2[\Gamma(\beta+1)]^2-1}{\Gamma(2\alpha+1)}\right)t^{2\alpha}\right]+b_1x^\beta,\hskip 1.5cm\mbox{if~ $0<\alpha\leq \dfrac{1}{2},$}\vspace{.4cm}\\
	\left[a_1+a_2t+\left(\dfrac{b_1^2[\Gamma(\beta+1)]^2-1}{\Gamma(2\alpha+1)}\right)t^{2\alpha}+\left(\dfrac{2b_2^2[\Gamma(\beta+1)]^2}{\Gamma(2\alpha+3)}\right)t^{2\alpha+2}\right.\\
	\left.+\dfrac{2b_1b_2[\Gamma(\beta+1)]^2}{\Gamma(2\alpha+2)}t^{2\alpha+1}\right]+\left[b_1+b_2t\right]x^\beta,\hskip 1.3cm\mbox{if~ $\dfrac{1}{2}<\alpha\leq 1$},
	\end{array}\right.
	\end{align}where $a_1,~a_2,~b_1$ and $b_2$ are arbitrary constants.\\
	Using initial conditions (\ref{IC7}) we get $a_1=e,a_2=1,b_1=\dfrac{1}{\Gamma(\beta+1)}$ and $b_2=-1.$	Hence the  exact solution of the IVP (\ref{EX7})-(\ref{IC7}) is
	
	\begin{eqnarray*}
		f(x,t)=\left\{\begin{array}{lcr}
			e+\dfrac{x^\beta}{\Gamma(\beta+1)},\hskip 5.7cm\mbox{if~ $0<\alpha\leq \dfrac{1}{2},$}\vspace{.4cm}\\
			\left[e+t+\dfrac{2[\Gamma(\beta+1)]^2}{\Gamma(2\alpha+3)}t^{2\alpha+2}-\dfrac{2\Gamma(\beta +1)}{\Gamma(2\alpha+2)}t^{2\alpha+1}\right]\\
			+\left[\dfrac{1}{\Gamma(\beta+1)}-t\right]x^\beta,\hskip4.5cm\mbox{if~ $\dfrac{1}{2}<\alpha\leq 1$}.
		\end{array}\right.
	\end{eqnarray*}\\
	\textbf{Type (II):}
	\noindent Consider another fractional generalization of  wave equation with constant absorption term
	\begin{equation}\label{EX72}
	\dfrac{\partial^{\alpha+1} f}{\partial t^{\alpha+1}}=N[f]=\dfrac{\partial^\beta }{\partial x^\beta}\left(f\dfrac{\partial^\beta f}{\partial x^\beta}\right)-1,~~\alpha,~\beta\in (0,1].
	\end{equation}
	From the discussion of Case (I), we know that $I_2=\mathfrak{L}\{1,x^\beta\}$ is an invariant subspace of $N[f]$. By Theorem \ref{THM1}, Eq. (\ref{EX72}) has the following solution.
	\begin{equation*}
	f(x,t)=K_0(t)+K_1(t)x^\beta,
	\end{equation*} where unknown functions $K_0(t)$ and $K_1(t)$ satisfy the following system of FDEs.
	\begin{align*}
	\dfrac{\rm d^{\alpha +1}K_0(t)}{\rm d t^{\alpha +1 }}&=K_1^2(t)[\Gamma(\beta+1)]^2-1,\\
	\dfrac{\rm d^{\alpha +1 } K_1(t)}{\rm d t^{\alpha +1 }}&=0.
	\end{align*}
	Solving these FDEs we obtain the following solution of Eq. (\ref{EX72}).
	\begin{align}\label{SOL72}
	f(x,t)=\left[a_1+a_2t+\left(\dfrac{b_1^2[\Gamma(\beta+1)]^2-1}{\Gamma(\alpha+2)}\right)t^{\alpha+1}+\left(\dfrac{2b_1b_2[\Gamma(\beta+1)]^2}{\Gamma(\alpha+3)}\right)t^{\alpha+2}\right.\nonumber\hskip-1.5cm\\
	\left.	+\left(\dfrac{2b_2^2[\Gamma(\beta+1)]^2}{\Gamma(\alpha+4)}\right)t^{\alpha+3}\right]+[b_1+b_2 t]x^\beta,
	\end{align}
	where $a_1, a_2, b_1$ and $b_2$ are arbitrary constants.\\
	
	\textbf{Note:} For $\alpha=1=\beta$, solutions (\ref{SOL71}) and (\ref{SOL72}) of two different fractional generalizations (\ref{EX7}) and (\ref{EX72}) respectively, reduce to
	\begin{align*}
	f(x,t)=a_1+a_2t+\left(\dfrac{b_1^2-1}{2}\right)t^{2}+\dfrac{b_1b_2}{3}t^{3}+\dfrac{b_2^2}{12}t^{4}+[b_1+b_2 t]x,
	\end{align*}
	which is a solution of the ordinary wave equation  with constant absorption term
	\begin{equation*}
	\dfrac{\partial^{2} f}{\partial t^{2}}=\dfrac{\partial^\beta }{\partial x^\beta}\left(f\dfrac{\partial^\beta f}{\partial x^\beta}\right)-1.
	\end{equation*}
	%Ex*************************************************************
	
	%Ex8*************************************************		
	\begin{example}
		Consider the following fractional generalization of Korteweg–de Vries equation (KdV) equation:
		\begin{equation}\label{EX8}
		\dfrac{\partial^{\alpha} f}{\partial t^{\alpha}}-6f\dfrac{\partial^\beta f}{\partial x^\beta}+\dfrac{\partial ^{2\beta}}{\partial x^{2\beta}}\left(\dfrac{\partial^\beta f}{\partial x^\beta}\right)=0,~\alpha\in (0,1) \backslash\{1/2\},~\beta\in (0,1].
		\end{equation}
	\end{example}
	\noindent Here $N[f]=6f\dfrac{\partial^\beta f}{\partial x^\beta}-\dfrac{\partial ^{2\beta}}{\partial x^{2\beta}}\left(\dfrac{\partial^\beta f}{\partial x^\beta}\right).$ $I_2=\mathfrak{L}\{1,x^\beta\}$ is an invariant subspace as
	\begin{equation*}
	N[k_0+k_1x^\beta]=6k_0k_1\Gamma(\beta+1)6k_1^2\Gamma(\beta+1)x^\beta \in I_2.
	\end{equation*}
	Hence in view of Theorem \ref{THM2}, Eq. (\ref{EX8}) has solution of the form\\
	$
	f(x,t)=K_0(t)+K_1(t)x^\beta,$ where $K_0(t)$ and $K_1(t)$ satisfy the following system of FDEs.
	\begin{eqnarray}
	\dfrac{\rm d^{\alpha }K_0(t)}{\rm d t^{\alpha }}&=&6K_0K_1\Gamma(\beta+1),\label{81}\\
	\dfrac{\rm d^{\alpha }K_1(t)}{\rm d t^{\alpha }}&=&6K_1^2\Gamma(\beta+1).\label{82}
	\end{eqnarray}
	Solving (\ref{81})-(\ref{82}) we get
	\begin{eqnarray*}
		K_0(t)=at^{-\alpha},
		K_1(t)=\frac{\Gamma(1-\alpha)t^{-\alpha}}{6\Gamma(1+\beta)\Gamma(1-2\alpha)}.
	\end{eqnarray*}Thus solution of Eq. (\ref{EX8}) becomes
	\begin{equation*}
	f(x,t)=at^{-\alpha}+\left[\frac{\Gamma(1-\alpha)t^{-\alpha}}{6\Gamma(1+\beta)\Gamma(1-2\alpha)}\right]x^{\beta},~a\in \mathbb{R}.
	\end{equation*}

	%Ex9**************************************************
	\begin{example}Consider the following non-linear space-time fractional dispersion equation.
		\begin{equation}\label{EX9}
		\dfrac{\partial^\alpha f}{\partial t^\alpha}=\dfrac{\partial^{\beta} }{\partial x^{\beta}}\left(	\dfrac{\partial^{\beta} }{\partial x^{\beta}}\left(\dfrac{\partial ^\beta }{\partial x^\beta}\left(\dfrac{f^2}{2}\right)\right)\right),~~\alpha,~\beta\in (0,1].
		\end{equation}
	\end{example}
	\noindent Here $I_2=\mathfrak{L}\{1,x^\beta,x^{2\beta}\}$ is an invariant subspace as
	\begin{equation*}
	N[k_0+k_1x^\beta,k_2x^{2\beta}]=\dfrac{\Gamma(4\beta+1)}{2~ \Gamma(\beta+1)}k_2^2x^{\beta}+k_1k_2\Gamma(3\beta+1)\in I_2.
	\end{equation*}
	We look for an exact solution of the form
	\begin{equation}\label{SOL9}
	f(x,t)=K_0(t)+K_1(t)x^\beta+K_2(t)x^{2\beta},
	\end{equation} where $K_0(t),~K_1(t)$ and $K_2(t)$ are unknown functions that are to be determined by solving the following system of FDEs.
	\begin{eqnarray}
	\dfrac{\rm d^\alpha K_0(t)}{\rm d t^\alpha}&=&\Gamma(3\beta+1)K_1(t)K_2(t)\label{91},\\
	\dfrac{\rm d^\alpha K_1(t)}{\rm d t^\alpha}&=&\dfrac{\Gamma(4\beta+1)}{2~ \Gamma(\beta+1)}K_2(t)^2,\label{92}\\
	\dfrac{\rm d^\alpha K_2(t)}{\rm d t^\alpha}&=&0.\label{93}
	\end{eqnarray}
	Eq.(\ref{93}) implies that $K_0(t)=$ constant $=c$ (say). Hence Eq. (\ref{92}) takes the form $	\dfrac{\rm d^\alpha K_1(t)}{\rm d t^\alpha}=c^2\dfrac{\Gamma(4\beta+1)}{2~ \Gamma(\beta+1)}$. Performing fractional integration on both sides, we obtain $K_1(t)=K_1(0)+\dfrac{c^2\Gamma(4\beta+1)}{2~ \Gamma(\beta+1)\Gamma(\alpha+1)}t^\alpha$. Proceeding on similar lines we get
	\begin{equation*}
	K_0(t)=K_0(0)+\dfrac{K_1(0)c\Gamma(3\beta+1)}{\Gamma(\alpha+1)}t^\alpha+\dfrac{c^3\Gamma(4\beta+1)\Gamma(3\beta+1)}{2~ \Gamma(\beta+1)\Gamma(2\alpha+1)}t^{2\alpha}.
	\end{equation*}
	Hence an exact solution of (\ref{EX9}) is
	\begin{eqnarray*}
		f(x,t)=\left[a+\dfrac{bc\Gamma(3\beta+1)}{\Gamma(\alpha+1)}t^\alpha+\dfrac{c^3\Gamma(4\beta+1)\Gamma(3\beta+1)}{2~ \Gamma(\beta+1)\Gamma(2\alpha+1)}t^{2\alpha}\right]\\
		+
		\left[b+\dfrac{c^2\Gamma(4\beta+1)}{2~ \Gamma(\beta+1)\Gamma(\alpha+1)}t^\alpha\right] x^\beta
		+cx^{2\beta},
	\end{eqnarray*}where $ a,~b$ and $c$ are arbitrary constants.\\
	
	Third order nonlinear time-fractional dispersion equation (corresponding to $\beta=1$) was studied in \cite{harris2014nonlinear} and the original model (corresponding to $\alpha=\beta=1$) was investigated by Galaktionov and Pohozaev \cite{galaktionov2008third}.
	%Ex10******************************************************	
	\begin{example}
		Consider the following fractional version of non linear heat equation:
		\begin{equation}\label{EX10}
		\dfrac{\partial^{\alpha} f}{\partial t^{\alpha}}=N[f]=\dfrac{\partial^{\beta} }{\partial x^{\beta} }\left(f\dfrac{\partial^{\beta} f}{\partial x^{\beta}}\right),~~\alpha,~\beta\in (0,1].
		\end{equation}
	\end{example}
	\noindent$I_2=\mathfrak{L}\{1,x^{\beta}\}$ is an invariant subspace of N[f] as
	\begin{equation*}
	N[k_0+k_1x^{\beta}]=\dfrac{\partial^{\beta} }{\partial x^{\beta} }\left(k_0k_1\Gamma(\beta +1) +k_1^2\Gamma(\beta +1)x^{\beta}\right)=k_1^2[\Gamma(\beta+1)]^2 \in I_2.
	\end{equation*}
	We consider exact solution of the form
	\begin{equation}\label{SOL10}
	f(x,t)=K_0(t)+K_1(t)x^{\beta},
	\end{equation} such that
	\begin{eqnarray}
	\dfrac{\rm d^\alpha K_0(t)}{\rm d t^\alpha}&=&K_1(t)^2[\Gamma(\beta+1)]^2,\label{101}\\
	\dfrac{\rm d^\alpha K_1(t)}{\rm d t^\alpha}&=&0.\label{102}
	\end{eqnarray}
	Solving the system of FDEs given by (\ref{101})-(\ref{102}) we obtain
	\begin{align*}
	K_1(t)&=b,\\ K_0(t)&=K_0(0)+b^2\dfrac{[\Gamma(\beta+1)]^2}{\Gamma(\alpha+1)}t^\alpha.
	\end{align*}
	Substituting values of $K_0(t)$ and $K_1(t)$ in Eq. (\ref{SOL10}) we get an exact solution as
	\begin{equation*}
	f(x,t)=\left[a+b^2\dfrac{[\Gamma(\beta+1)]^2}{\Gamma(\alpha+1)}t^\alpha\right] +bx^{\beta},~a, b \in \mathbb{R}.
	\end{equation*}

	%Ex11*******************************************************
	\begin{example}
		Consider the following fractional PDE:
		\begin{equation}\label{EX11}
		\dfrac{\partial^{\alpha} u}{\partial t^{\alpha}}=\dfrac{\partial^{\beta} u}{\partial x^{\beta}}+\left(\dfrac{\partial^{\beta} u}{\partial x^{\beta}}+u\right)^{-1},~~\alpha,~\beta\in (0,1].
		\end{equation}
	\end{example}
	\noindent Here $N[u]=\dfrac{\partial^{\beta} u}{\partial x^{\beta}}+\left(\dfrac{\partial^{\beta} u}{\partial x^{\beta}}+u\right)^{-1}.~I_2=\mathfrak{L}\{1,E_\beta(-x^\beta)\}$ is an invariant subspace as $
	N[k_0+k_1E_\beta(-x^\beta)]=-k_1E_\beta(-x^\beta)+\dfrac{1}{k_0}\in I_2,~k_0\neq 0.$\\
	Consider the exact solution of the form
	\begin{equation}\label{SOL11}
	f(x,t)=K_0(t)+K_1(t)E_\beta(-x^\beta),~K_0(t)\neq 0,
	\end{equation} where $K_0(t)$ and $K_1(t)$ satisfy the following system of FDEs
	\begin{eqnarray}
	\dfrac{\rm d^{\alpha} K_0(t)}{\rm d t^\alpha}&=&\frac{1}{K_0(t)},\label{111}\\
	\dfrac{\rm d^{\alpha} K_1(t)}{\rm d t^\alpha}&=&-K_1(t).\label{112}
	\end{eqnarray}
	Assuming that solution of Eq. (\ref{111}) is of the form $K_0(t)=\eta t^\gamma$ we get $K_0(t)=\pm\left(\frac{\Gamma(1-\alpha/2)}{\Gamma(1+\alpha/2)}\right)^{1/2}t^\frac{\alpha}{2}$ and using Laplace transform technique for Eq. (\ref{112}) we deduce $K_1(t)=K_1(0)E_{\alpha}(-t^\alpha)$ which leads to the following solution of Eq. (\ref{EX11}).
	\begin{equation*}
	f(x,t)= \left[\pm\left(\frac{\Gamma(1-\alpha/2)}{\Gamma(1+\alpha/2)}\right)^{1/2}t^{\frac{\alpha}{2}}\right]+[bE_{\alpha}(-t^\alpha)]E_\beta(-x^\beta),~b=K_1(0).
	\end{equation*}

	%Ex10*********************************************************

	\section{Conclusions}
	Present article develops invariant subspace method for solving nonlinear time and space fractional partial differential equations. In this method FPDEs are reduced to systems of fractional differential equations which can be solved by existing analytic methods. In general fractional PDEs admit more than one invariant subspace, each of which yields an exact solution. In fact linear fractional PDEs admit infinitely many invariant subspaces (cf. (\ref{EX1}) and  (\ref{EX3})). Invariant subspace method is also used to derive closed form solutions of time and space fractional PDEs along with certain kinds of initial conditions. It is observed that invariant subspace method is very effective and rigorous tool for finding exact solutions of wide class of linear/ non linear FPDEs.
	
	%%%%%%%%%%%%%%%%%%%%%%%%%%%%%%%%%%%%%%%%%%%%%%%%%
	\section*{Acknowledgements}
	
	Sangita Choudhary acknowledges the National Board for Higher Mathematics, India, for the award of Research Fellowship.
	
	%%%%%%%%%% References %%%%%%%%%%%%%%%%%%%%%%%%%%%%%%%%
	%%%% arranged in ALPHABETIC ORDER of Authors' Families
	%%%% for articles, insert also DOI numbers if available
	%\bibliographystyle{abbrv}
	%\bibliography{Bib}

\begin{thebibliography}{99}
		\normalsize
		
		
		\bibitem{adomian2013solving}
		G.~Adomian,	\emph{Solving Frontier Problems of Physics: The Decomposition Method}. Springer Science and Business Media, Netherland \textbf{60} (2013).
		
		\bibitem{daftardar2008solving}
		V.~Daftardar-Gejji and S.~Bhalekar, Solving fractional diffusion-wave equations using a new iterative
		method. \emph{Fract. Calc. Appl. Anal.} \textbf{11}, No 2 (2008), 193--202.
		
		\bibitem{daftardar2008solvingmulti}
		V.~Daftardar-Gejji and S.~Bhalekar, Solving multi-term linear and non-linear diffusion--wave equations of
		fractional order by adomian decomposition method.
		\emph{Appl. Math. Comput.} \textbf{202}, No 1 (2008), 113--120. 
		
		\bibitem{debnath2003recent}
		L.~Debnath, Recent applications of fractional calculus to science and
		engineering. \emph{Internat. J. Math. Math. Sci.} \textbf{2003}, No 54 (2003), 3413--3442.
		
		\bibitem{diethelm2010analysis}
		K.~Diethelm, \emph{The Analysis of Fractional Differential Equations}. Springer, New York (2010).
		
		\bibitem{galaktionov1995invariant}
		V.~A. Galaktionov, Invariant subspaces and new explicit solutions to evolution equations	with quadratic nonlinearities. \emph{ Proc. Roy. Soc. Edin. Sec. A} \textbf{125}, No 2 (1995), 225--246.
		
		\bibitem{galaktionov2006exact}
		V.~A. Galaktionov and S.~R. Svirshchevskii, \emph{Exact Solutions and Invariant Subspaces of Nonlinear Partial Differential Equations in Mechanics and Physics}. CRC Press, London (2006).
		
		\bibitem{galaktionov2008third}
		V.~A. Galaktionov and S.~Pohozaev, Third-order nonlinear dispersive equations: Shocks, rarefaction, and	blowup waves.\emph{ Comput. Math. Math. Phys.} \textbf{48}, No 10 (2008), 1784--1810.
		
		\bibitem{gazizov2013construction}
		R.~Gazizov and A.~Kasatkin, Construction of exact solutions for fractional order differential equations by the invariant subspace method. \emph{Comput. Math. Appl.} \textbf{66}, No 5 (2013), 576--584.
		
		\bibitem{harris2013analytic}
		P.~A. Harris and R.~Garra, Analytic solution of nonlinear fractional burgers-type equation by invariant subspace method. \emph{Nonlinear Stud.} \textbf{20}, No 4 (2013), 471-481
		
		\bibitem{harris2014nonlinear}
		P.~A. Harris and R.~Garra, Nonlinear time-fractional dispersive equations. \emph{Commun. Appl. Indus. Math.} \textbf{6}, No 1 (2014).
		
		\bibitem{he1999homotopy}
		J.~H. He, Homotopy perturbation technique. \emph{Comput. Methods Appl. Mech. Engirg.} \textbf{178}, No 3 (1999), 257--262.
		
		\bibitem{ma2012refined}
		W.~X. Ma, A refined invariant subspace method and applications to evolution equations. \emph{Sci. China Math.} \textbf{55}, No 9 (2012), 1769--1778.
		
		\bibitem{podlubny1998fractional}
		I.~Podlubny, \emph{Fractional differential equations}. Academic press, New york (1998).
		
		\bibitem{sahadevan2015invariant}
		R.~Sahadevan and T.~Bakkyaraj, Invariant subspace method and exact solutions of certain nonlinear time fractional partial differential equations. \emph{Fract. Calc. Appl. Anal.} \textbf{18}, No 1 (2015), 146--162; DOI:10.1515/fca-2015-0010.
		
		\bibitem{sahadevan2016exact}
		R.~Sahadevan and P.~Prakash, Exact solution of certain time fractional nonlinear partial differential equations. \emph{Nonlinear Dyn.} (2016), 1--15.
		%; DOI 10.1007/s11071--016--2714--4.
		
		\bibitem{svirshchevskii1995lie}
		S.~Svirshchevskii, Lie-b{\"a}cklund symmetries of linear odes and generalized separation of variables in nonlinear equations. \emph{Phys. lett. A} \textbf{199}, No 5-6 (1995), 344--348.
		
		\bibitem{svirshchevskii1996invariant}
		S.~R. Svirshchevskii, Invariant linear spaces and exact solutions of nonlinear evolution	equations. \emph{J. Nonlinear Math. Phys.} \textbf{3}, No 1-2 (1996), 164--169.
	\end{thebibliography}

	 %%%%%%%%%%%%%%%%%%%%%%%%%%%%%%%%}

\end{document}